\documentclass[a4paper,12pt]{article}

\usepackage{amsmath,amssymb,amsfonts,amsthm,enumerate,thmtools,hyperref,etoolbox,enumitem}
\usepackage{tikz}
\usetikzlibrary{calc,shapes.misc}
\usepackage[symbol]{footmisc}

\newcommand{\eps}{\varepsilon}
\newtheorem{theorem}{Theorem}
\newtheorem{lemma}[theorem]{Lemma}
\newtheorem{remark}[theorem]{Remark}
\newtheorem{definition}[theorem]{Definition}

\newtheorem{corollary}[theorem]{Corollary}
\newtheorem{claim}[theorem]{Claim}

\title{On the relation of separability, bandwidth and embedding}
\author{
	B\'ela Csaba\thanks{Bolyai Institute, University of Szeged, Hungary, {email: bcsaba@math.u-szeged.hu}. Partially supported by NKFIH Fund No.~KH 129597 and by NKFIH Fund No.~SNN 117879.},
	B\'alint V\'as\'arhelyi\thanks{Bolyai Institute, University of Szeged, Hungary, email: mesti@math.u-szeged.hu} 
}

\begin{document}

\maketitle	

\begin{abstract}
In this paper we construct a class of bounded degree bipartite graphs with a small separator and large bandwidth.
Furthermore, we also prove that graphs from this class are spanning subgraphs of graphs with minimum degree just slightly larger than $n/2$. 
\end{abstract}

\section{Introduction} 

One of the most basic questions in graph theory is, given graphs $H$ and $G,$ to decide whether $H$ is a subgraph of $G.$ 
If so, we also say that $H$ can be {\it embedded} into $G,$ and call $G$ the {\it host graph.}




In this paper we construct a class of bounded degree  bipartite graphs that have small separator and large bandwidth, and prove that the graphs of this
class are spanning subgraphs of $n$-vertex graphs that have minimum degree just slightly larger than $n/2.$ We also show that using earlier methods these 
graphs cannot be embedded in general into host graphs with such small minimum degree.

The structure of the paper is as follows. First we give an brief overview of the area, finishing with a list of our new results. Next we define an infinite class of bounded degree graphs having small separators and very large bandwidth. Finally, we embed such graphs using the Regularity lemma -- Blow-up lemma method.

\section {Overview of the area and our new results}

In this paper we consider only simple graphs. We use standard graph theory notation. In particular, if $F=(V, E)$ is a graph then 
the degree of a vertex $v\in V(F)$ is denoted by $deg_F(v),$ or just $deg(v),$ if $F$ is understood from the context. The number of vertices of $F$
is denoted by $|F|,$ and we let $e(F)=|E(F)|.$ Given $A\subset V(F)$ we let $F[A]$ denote the subgraph of $F$ that is spanned by the vertices in $A.$ If $A, B\subset V(F)$ such that $A\cap B=\emptyset$ then $F[A, B]$ denotes the bipartite subgraph of $F$ which contains precisely the edges with one endpoint in $A$ and the other endpoint in $B.$
Given $S\subset V(F)$ then $N(S)$ denotes the set of those vertices that have at least one neighbor in $S.$ The maximum degree of $F$ is denoted by $\Delta(F),$ the minimum degree of $F$ is denoted by $\delta(F).$

Perhaps the most cited result in extremal graph theory is the celebrated theorem of Dirac, stating that if the minimum degree of a graph on $n\ge 3$
vertices is at least $n/2,$ then the graph contains a Hamilton cycle. This result was generalized in various ways. In a {\it Dirac-type} embedding problem one gives a
lower bound on the minimum degree of the host graph $G,$ and finds conditions for a graph $H$ on the same number of vertices which guarantees that $H\subset G.$  

The famous Bollob\'as-Eldridge-Catlin conjecture~\cite{BE, C} is a Dirac-type question, in which we have a bound for the maximum degree of the graph to be embedded.
It asserts that if $G$ and $H$ are graphs on $n$ vertices, and $\delta(G)\ge (1-1/(\Delta(H)+1))n$ then $H\subset G.$ 
It is still open in general, only some special cases were solved (for large $n$), e.g., the cases $\Delta(H)=2$~\cite{AignerBrandt}, $\Delta(H)=3$~\cite{CsSSz}, and when
$\Delta(H)$ is bounded and $H$ is bipartite~\cite{BEparos}. Kaul, Kostochka and Yu proved an approximation result in~\cite{KKY}. 

One may impose other kind of restrictions on $H$ and obtains still hard problems. For example, one may upper bound the so-called bandwidth of $H,$   this guarantees that $H$ is ``far from being an expander". 
The bandwidth of a graph $H$ is denoted by $bw(H)$ and is defined to be the smallest positive integer $b$ such that there exists a
labelling of the vertices of $V(H)$ by the numbers $\{1,\ldots, n\}$ such that the labels of every pair of adjacent vertices
differ by at most $b.$

Note that a Hamilton path has bandwidth 1, a Hamilton cycle has bandwidth 2. Expander graphs have large, linear bandwidth, a star on $n$ vertices 
has bandwidth $n/2,$ a complete graph has bandwidth $n-1.$ 

One of the most important open problems concerning the bandwidth was a conjecture by Bollob\'as and Koml\'os~\cite{Komlos}. This conjecture was proved by B\"ottcher, Schacht and Taraz~\cite{BST}\footnote{We remark, that the case of $k=2$ was first proved by Abbasi~\cite{Sarmad}, and is also a special case of a result by the author~\cite{wellsep}. Some of the ideas of the latter proof are used in the present paper as well.} using deep tools, in particular the Regularity lemma and the proof of the celebrated P\'osa-Seymour conjecture by Koml\'os, S\'ark\"ozy and Szemer\'edi~\cite{PosaSeymour}.

\begin{theorem}\label{BollKomlos}
For every $D, k\in \mathbb{N}$ and $\varepsilon >0$ there exists $\beta>0$ such
that the following holds. Every $n$-vertex graph $G$ having minimum degree
$\delta(G)\ge (1 - 1/k+\varepsilon)n$ contains all $k$-chromatic $n$-vertex graphs of maximum degree
at most $D$ and bandwidth at most $\beta n$ as subgraphs.
\end{theorem}

B\"ottcher~\cite{Btezis} and B\"ottcher et al.~\cite{BPstb} went further and explored relations of bandwidth with other notions, like separability. 
We say that an $n$-vertex graph $H$ is $\gamma$-separable if there exists a separator set $S\subset V(H)$ with $|S|\le \gamma n$ such that
every component of $H-S$ has at most $\gamma n$ vertices. B\"ottcher et al.~\cite{BPstb} observed that bandwidth and separability are closely related: they proved the Sublinear Equivalence Theorem. This states that, roughly speaking, in bounded degree graphs
sublinear bandwidth implies the existence of a sublinear sized separating set and vice versa.  

It is easy to see that there are bounded degree graphs having linearly large bandwidth: it is well-known that a random $l$-regular graph with $l\ge 3$
has large bandwidth with positive probability. However, such random graphs do not have small separators. 
One of our results shows that when the separating set has small (but not very small) linear size, the bandwidth can be very large even for bounded degree graphs.

\begin{theorem}\label{fo1}
Let $r\ge 35$ and $t\ge 1$ be integers and set  $\gamma=\gamma(r)=  1/(8r2^r).$ Then one can construct an infinite class of graphs 
$\mathcal{H}_{r, t}$ such that every element $H$ of $\mathcal{H}_{r, t}$ has a separator set of size $\le \gamma |H|,$ has bandwidth at least $0.3|H|/(2t+4),$ moreover, $\Delta(H)=O(1/\gamma).$
\end{theorem}

There is a recent interest in embedding graphs with sublinear bandwidth. Staden and Treglown~\cite{ST} embed graphs on $n$ vertices with sublinear bandwidth into locally
dense graphs on $n$ vertices having minimum degree at least $(1/2 + o(1))n.$ Condon, Kim, K\"uhn and Osthus~\cite{CKKO} find an approximate decomposition of a certain class of graphs into graphs that have sublinear bandwidth.

The result of Knox and Treglown~\cite{KT} is particularly interesting for us. They embedded bounded degree graphs with sublinear bandwidth into so called robust expanders. 
Let $0< \nu \le \tau <1.$ Assume that $G$ is a graph of order $n$ and  
$S\subset V(G).$ The $\nu$-robust neighborhood $RN_{\nu, G}(S)$ of $S$ is the set of vertices $v\in V(G)$ such that $|N(v)\cap S|\ge \nu n.$ We say that $G$ is a robust $(\nu, \tau)$-expander
if $|RN_{\nu, G}(S)|\ge |S|+\nu n$ for every $S\subset V(G)$ such  that $\tau n\le |S|\le (1-\tau)n.$ 

We will also show that 
elements of $\mathcal{H}_{r, t}$ (the graph class of Theorem~\ref{fo1}) cannot be embedded into arbitrary robust expanders. 
However, if an $n$-vertex graph $G$ has minimum degree slightly larger than $n/2,$ then it contains the elements of $\mathcal{H}_{r, t}$ as spanning 
subgraphs. We will prove the following. 


\begin{theorem}\label{fo2}
Let $r\ge 35$ and $t\ge 1$ be integers and set  $\gamma=\gamma(r)=  1/(8r2^r).$ Then there exists an $n_0=n_0(\gamma)$ such that the following holds. 
Assume that $n\ge n_0$ and $G$ is an $n$-vertex graph having minimum degree $\delta(G)\ge (1/2+3\gamma^{1/3})n.$ If $H\in \mathcal{H}_{r,t}$ is a graph on 
$n$ vertices, then $H\subset G.$ 
\end{theorem}

A standard example shows that one cannot significantly reduce the minimum degree of $G$ in the above theorem. Let $G$ be the union of two complete graphs on, say,
$n/2+\gamma n/100$ vertices that share $\gamma n/50$ vertices. Clearly, $\delta(G)=n/2+\gamma n/100-1.$ It is an easy exercise to prove\footnote{Of course, one first needs the definition of $\mathcal{H}_{r,t}$ for this proof, which is given in Section~\ref{harmadik} soon.} that if $H\in \mathcal{H}_{r,t}$ and $|H|=|G|=n,$ then $H\not\subset G.$ We leave this proof for the reader. 

The proof of Theorem~\ref{fo2} will rely heavily on the proof method of~\cite{wellsep} and an important result of Fox and Sudakov~\cite{FS}. Let us remark that in~\cite{wellsep} 
the size of the separator set was $o(n),$ and therefore, by the Sublinear Equivalence theorem, the bandwidth was also $o(n).$ This time the separator set is quite large 
compared to previous results. 


Finally, let us mention a tightness result for Theorem~\ref{BollKomlos} by Abbasi~\cite{Sarmad}. He proved that for infinitely many $n$ there exist graphs $G$ and $H$ on $n$ vertices such that $\delta(G)=(1/2 +\eta)n$ and $bw(H)\le 4\eta n,$ still $H\not\subset G.$ So in general if one wants to be able to embed every graph with bandwidth $\eta n,$ then the minimum degree bound for the host graph must  
be larger than $n/2 +  \eta n/4.$ 

\section{Construction of $\mathcal{H}_{r,t}$ and proof of Theorem~\ref{fo1}}\label{harmadik}

In order to exhibit the infinite family of graphs $\mathcal{H}_{r,t}$ we first need to construct certain kind of bipartite expander graphs. 
We begin with defining a bipartite graph $F$ with vertex classes $V_1$ and $V_2$ such that $|V_1|=|V_2|=k$ and $F$ has relatively good expansion properties. Our construction of $F$ relies on the existence of so called Ramanujan graphs\footnote{We remark that here one may as well work with random regular bipartite graphs instead of explicit constructions.}: an $r$-regular (non-bipartite) graph $U$ is a Ramanujan graph if $\lambda \le 2\sqrt{r-1},$ where $\lambda$ is the second largest in absolute value of the eigenvalues of $U$ (since $U$ is $r$-regular, the largest eigenvalue is $r$). Lubotzky, Phillips and Sarnak~\cite{LPS}, and independently Margulis~\cite{Marg}, constructed for every $r=p+1$ where $p\equiv 1\mod 4$  
infinite families of $r$-regular graphs with second largest eigenvalues at most $2\sqrt{r-1}.$ 
We need a fact about these graphs, a lower bound for the number of edges between subsets of $U.$

\begin{lemma}\label{elszam}
Let $U$ be a graph as above. Then for every two subsets $A, B \subset V(U)$ where $|A|=ak$ and $|B|=bk$ we have
$$|e(A, B)-abrk|\le 2\sqrt{r-1}\sqrt{ab}k.$$
\end{lemma}

The proof of Lemma~\ref{elszam} can be found for example in~\cite{AS}.

\begin{corollary}\label{kov}
Let $U$ be an $r$-regular Ramanujan graph on $k$ vertices with $r\ge 35.$ Let us assume that $A, B\subset V(U)$ with $|A|=|B|=k/3$ and $A\cap B=\emptyset.$ Then $e(A, B)\ge 1.$
\end{corollary}

\begin{proof}
 It is easy to see that the expression of Lemma~\ref{elszam} gives a lower bound for $e(A, B)$ which is monotone increasing in $r.$ Hence it is sufficient to apply Lemma~\ref{elszam} with $r=35$ and $a=b=1/3.$
Straightforward computation gives what was desired.
\end{proof}

We are ready to discuss the details of the construction of $F.$ Given an $r$-regular Ramanujan graph $U$ with $r\ge 35$ the vertex classes of $F$ will be copies of $V(U)$: for every $x\in V(U)$ we have two copies of it, $x_1\in V_1$
and $x_2\in V_2.$ For every $xy\in E(U)$ we include the edges $x_1y_2$ and $x_2y_1$ in $E(F).$ Finally, for every $x\in V(U)$ we will also have the edge $x_1x_2$ in $E(F).$ Observe that $F$
is an $(r+1)$-regular bipartite graph.
The following claims are crucial for the construction of $\mathcal{H}_{r,t}.$ 

\begin{claim}\label{elszam2}
Let $A\subset V_1$ and $B\subset V_2$ be arbitrary such that $|A|=|B|=k/3.$ Then $e_F(A, B)\ge 1.$
\end{claim}

\begin{proof}
If there exists $x\in V(U)$ such that $x_1\in A$ and $x_2\in B$ then we are done since every $x_1x_2$ edge is present in $F.$ If there is no such $x\in V(U)$ then we can apply Corollary~\ref{kov}
and obtain what is desired.
\end{proof}

\begin{claim}\label{expanzio2}
For every $A\subset V_1$ we have $|N_F(A)|\ge |A|.$ Analogous statement holds for any subset $B\subset V_2.$ 
\end{claim}

\begin{proof}
The claim easily follows from the fact that we included a perfect matching in $F$ when for every $x\in V(F)$ we added the $x_1x_2$ edge to $E(F).$ 
\end{proof}

Observe that we can construct a bipartite graph $F$ with $|V(F)|=2k$ having the above properties whenever there exists a Ramanujan graph $U$ with $|V(U)|=k,$ for the latter we also assume that $r\ge 35.$ 
Thus, there exists an infinite sequence of $\{F_i\}_{i=1}^{\infty}$ graphs on increasing number of vertices, 
say, $F_i$ has $2k_i$ vertices.

We are ready to define $\mathcal{H}_{r,t}.$ Each graph from this class
is $\gamma$-separable where $\gamma=\gamma(r)$ can be relatively  small as we will see soon. Still, the bandwidth of each of them is very large. Hence, $\mathcal{H}_{r,t}$ demonstrates that in spite of sublinear equivalence of separability and bandwidth, there is no {\it linear} equivalence.

The construction of $\mathcal{H}_{r,t}$ is somewhat specific, we do it with foresight as our goal is not only to further explore the relation of separability and bandwidth but also to be able to embed the elements of $\mathcal{H}_{r,t}$ later.

\begin{definition} Let $n, m \in \mathbb{N}$ be sufficiently large, $r\ge 35,$ $t\ge 1$ be integers, and set $\gamma=\gamma(r)=  1/(8r2^r).$ Let $F_i$ be the $(r+1)$-regular 
bipartite graph on $2k_i$ vertices given above such that
$k_i$ is the largest for which $\gamma n\ge 2k_i.$
	The elements of $\mathcal{H}_{r,t}$ are constructed as follows. Given $n$ we let $H=(A,B;E)\in \mathcal{H}_{r,t}$ to be the following bipartite graph. 

	\begin{enumerate}
                     \item  $V(H)=A\cup B$ and $|H| = |A\cup B| = n,$
		\item let $S\subset V$ such that  $S= S_A\dot{\cup} S_B$ and $|S_A| = |S_B| = k_i$, 
		\item $E(H[S_A])=E(H[S_B])=\emptyset$ and $H[S_A, S_B] = F_i,$ 
		\item $D=\Delta(H)=O(r2^r),$
		\item $H-S$ has exactly $|S|$ isomorphic components; each component is a tree $T$ which contains a path on $t$ vertices, and one of its endpoints, the {\em last vertex}, has $D-1$ leaves attached to it; we call the other endpoint of this path the {\em first vertex}.

		\item every $x\in S$ has a unique neighbor $y$ in $V-S$ which is the first vertex of one of the tree components of $H-S,$ moreover, every first vertex has exactly one neighbor in $S.$ 
	\end{enumerate}
\end{definition}

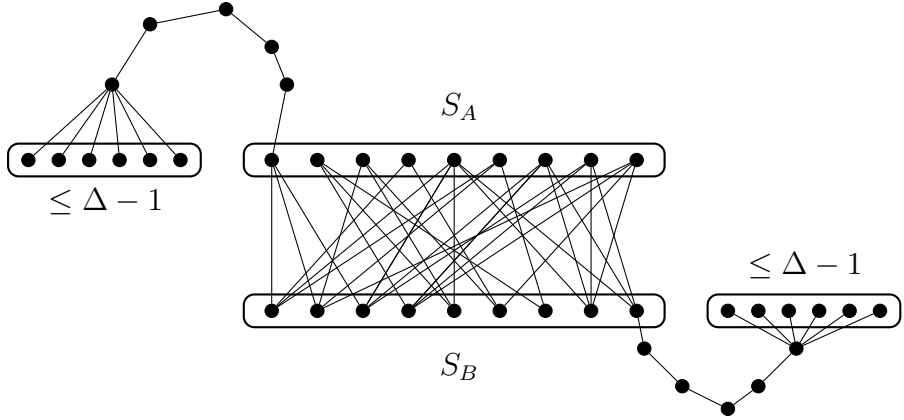
\begin{figure}[h!]
	\centering
	\begin{tikzpicture}[vertex/.style={draw,fill,circle,inner sep=0pt,minimum size=5pt}]

\foreach \x in {1,...,9}{
	\foreach \t in {0,2}
	{
		\node[vertex] (s\x\t) at (\x*0.6,\t)  {};
	}
}

\foreach \x/\t in {4/8,1/4,9/7,3/8,9/5,3/5,2/3,4/7,1/1,4/9,5/2,9/8,4/7,2/9,6/9,7/6,1/6,1/5,2/6,3/1,8/8,6/4,8/5,3/5,8/9,5/3,7/2,4/2,6/3,5/5,2/1,8/7,3/7}
{
	\draw (s\x0) -- (s\t2);
}

\draw[thick,rounded corners]     ($(s10.north west)+(-0.3,0.15)$) rectangle ($(s90.south east)+(0.3,-0.15)$) node[below,yshift=-0.2cm,xshift=-2.7cm]{$S_B$};
\draw[thick,rounded corners]     ($(s12.north west)+(-0.3,0.15)$) rectangle ($(s92.south east)+(0.3,-0.15)$) node[above,yshift=0.6cm,xshift=-2.7cm] {$S_A$};

\foreach \x/\t/\i in {0.8/3/0,0.6/3.5/1,0/4/2,-1/3.8/3,-1.5/3/4}
{
	\node[vertex] (x\i) at (\x,\t) {};
}
\draw (s12) -- (x0) -- (x1) -- (x2) -- (x3) -- (x4);

\foreach \x in {1,...,6}
{
	\node[vertex] (a\x) at (\x*0.4-3,2) {};
	\draw (x4) -- (a\x);
}
\draw[thick,rounded corners]     ($(a1.north west)+(-0.2,0.15)$) rectangle ($(a6.south east)+(0.2,-0.15)$) node[below,xshift=-1.25cm] {$\leq\Delta-1$};

\foreach \x/\t/\i in {5.5/-0.5/0,6/-1/1,6.6/-1.3/2,7/-1/3,7.5/-0.5/4}
{
	\node[vertex] (y\i) at (\x,\t) {};
}
\draw (s90) -- (y0) -- (y1) -- (y2) -- (y3) -- (y4);

\foreach \x in {1,...,6}
{
	\node[vertex] (b\x) at (\x*0.4+6.2,0) {};
	\draw (y4) -- (b\x);
}
\draw[thick,rounded corners]     ($(b1.north west)+(-0.2,0.15)$) rectangle ($(b6.south east)+(0.2,-0.15)$) node[above,xshift=-1.25cm,yshift=0.5cm] {$\leq\Delta-1$};

\end{tikzpicture}
\caption{The separator set $S$ and the way the components of $H-S$ are connected to $S$}\protect\label{fig1}
\end{figure}

Note that $S$ is a separator set of $H$ with $|S|=2k_i\approx \gamma n,$ every component of $H-S$ has less than $t+D$ vertices. From this one can easily obtain the bound $D< 3n/k_i\approx 6/\gamma.$ We remark that when $t=1$ then 
the first and the last vertex in every path are the same. 
The following lemma is crucial for bounding the bandwidth of $H\in \mathcal{H}_{r,t}.$

\begin{lemma} \label{shortpath}
Let $H$ be an element of $\mathcal{H}_{r,t}$ on $n$ vertices. Assume that $X,Y\subset V(H)$ with $|X|, |Y|\geq 0.35n$ and $X\cap Y=\emptyset.$ Then there exists an $x\in X$ and a $y\in Y$ such that the distance of $x$ and $y$ 
is at most $2t + 4.$
\end{lemma}

\begin{proof}

Denote the vertices of $H-S$ closer to $S_A$ by $A^*$, and analogously, the vertices of $H-S$ closer to $S_B$ by $B^*.$ By the construction of $H$ we have $|A^*| = |B^*| \ge (1-\gamma)n/2$.
Note that $\gamma<0.01$ since $r\ge 35.$ Hence we have that $|X-S|\geq 0.34n$ and $|Y-S|\geq 0.34n$. Thus, either $|X\cap A^*| \geq |A^*|/3$ or $|X\cap B^*|\geq |B^*|/3$. Without loss of generality, suppose the former. This also implies that at least $1/3$ of the components of $A^*$ have vertices in $X.$ Denote the first vertices in these components by $X^*.$ 
We also let $Y^*_A$ denote the set of first vertices of those components in $A^*$ that have at least one vertex from $Y,$ and analogously, $Y^*_B$ denotes the set of first vertices of those components in $B^*$ that have at least one vertex from $Y.$

Let $X_A=N(X^*)\cap S,$ $Y_A=N(Y^*_A)\cap S$ and $Y_B= N(Y^*_B)\cap S.$ Using these notations we have that $|X_A|\ge k/3$ and either $|Y_A|\geq {k/3}$, or $|Y_B|\geq {k/3}.$

If $|Y_B|\geq {k/3}$, then by Claim~\ref{elszam2} there is an edge $z_1z_2$ between $X_A$ and $Y_B,$ and therefore we have a path $xv_s \ldots v_1 z_1z_2 u_1\ldots u_q y$ of length $\le 2t + 3$, 
where $x\in X, y\in Y, v_i\in A^*$, $z_1\in X_A$, $z_2 \in Y_B$, $u_i\in B^*$. 
	
If $|Y_B|<\frac{k}{3}$, then $|Y_A| \geq {k/3}.$ Let $Y'_B=N(Y_A)\cap S.$ Claim~\ref{expanzio2} implies that 
$|Y'_B|\ge |Y_A| \ge k/3,$ so by Claim~\ref{elszam2} 
$H$ has an edge $z_2z_1$ between $Y_B'$ and $X_A.$ Thus, we have a path $xv_s \ldots v_1 z_1z_2z_3\linebreak[1] u_1\ldots u_q y$ of length $\le 2t + 4$, where $x\in X, y\in Y, v_i \in A^*, u_i\in B^*, z_1\in X_A$, $z_3 \in Y_A$ and $z_2 \in Y_B'.$ 
\end{proof}

\begin{corollary}
	Let $H$ be an element of $\mathcal{H}_{r,t}$ on $n$ vertices. Then the bandwidth of $H$ is at least $\frac{0.3n}{2t +4}$.
	\label{corollary:large-bandwidth}
\end{corollary}

\begin{proof} Take an arbitrary ordering $\mathcal{P}$ of the vertices of $H.$ Let $X$ be the first $0.35n$ vertices, and 
$Y$ be the last $0.35n$ vertices of $\mathcal{P}$. Using Lemma~\ref{shortpath} there is an $x\in X$ and an $y\in Y$ such that the distance of $x$ and $y$ is at most $2t + 4$. Their distance in $\mathcal{P}$ is at least $0.3n.$ Thus at least one of the edges of the shortest path between $x$ and $y$ must have ``length'' at least $\frac{0.3n}{2t + 4},$ from which the bound for the bandwidth follows immediately.
\end{proof}

 With this we proved Theorem~\ref{fo1}.
Observe that choosing $t=1$ results in graphs having bandwidth at least $3n/60=n/20$ while being $\gamma$-separable.
It is easy to see that, using the above ideas, one can define a much wider class of graphs having very large bandwidth that
are also $\gamma$-separable. Our main goal, however, is not only to construct but also to be able to embed such graphs.

\medskip

Recall the notion of robust expanders.
As we mentioned in the introduction, Knox and Treglown~\cite{KT} embedded spanning subgraphs of sublinear bandwidth into robust expanders. 
The following example shows that graphs of $\mathcal{H}_{r,t}$ not only have very large bandwidth,
these graphs are not necessarily subgraphs of robust expanders. Hence, in the theorem of Knox and Treglown one cannot replace sublinear bandwidth by 
$\gamma$-separability, unless $\gamma$ is very small. In fact the proofs of~\cite{BPstb} and~\cite{KT} works only
in case $\gamma\le 1/\ell,$ where $\ell$ denotes the number of clusters in a sufficiently large graph after applying the Regularity lemma with some small parameter $\varepsilon>0.$
It is known (see~\cite{Gowers}) that $\ell$ is bounded from below by a tower function of $1/\varepsilon.$  

Next we us construct a robust expander as follows. Let $G=(V,E)$ be a graph on $n$ vertices such that  the vertex set of $G$ is 
$V = A_1\dot{\cup}A_2\dot{\cup}\cdots\dot{\cup}A_{100}$, here $|A_i|=\frac{n}{100}$ for every $1 \leq i\le 100.$
The edges of $G$ are defined as follows: $G[A_i, A_{i+1}]$ is a complete bipartite graph for every $1 \le i\le 99,$ and 
and $G[A_1], G[A_{100}]$ are complete graphs on $n/100$ vertices. 

From the following remark one can easily see that $G$ is a $(\nu, \tau)$-robust expander, where $1/n<\nu$ and $\tau\ge 100\nu.$
Let $1<i<100$ and $Q\subset V(G)$ be an arbitrary set. Then every vertex in $A_{i-1}\cup A_{i+1}$ has $|Q\cap A_i|$ neighbors in $Q\cap A_i.$
Furthermore, every vertex in $A_1\cup A_2$ has at least $|Q\cap A_1|-1$ neighbors in $Q\cap A_1,$ unless $|Q\cap A_1|\le 1.$ Analogous  
statement holds for $A_{100}.$ 

\begin{lemma}	
	Let $H$ be a graph from $\mathcal{H}_{r,t}$ on $n$ vertices and let $G$ be as above. Then $H\not\subset G$ if $t\leq 13.$
\end{lemma}

\begin{proof}
	Let $X=A_1\cup \ldots \cup A_{35}$ and $Y=A_{66}\cup \ldots \cup A_{100}.$ If one embeds $H$ into $G,$ then between the 
vertices that were mapped onto some vertex of $X$ and those that were mapped onto some vertex of $Y$ there is a path having length
at most $2t+4.$ In $G$ the shortest path between $X$ and $Y$ has length $31.$ Since $2t+4<31$ whenever $t\le 13,$ we proved what
was desired.
\end{proof}

\section{Proof of Theorem~\ref{fo2}}

The proof of Theorem~\ref{fo2} is very similar to the proof of the main result of~\cite{wellsep}.
First let us state a special case of the main theorem of~\cite{wellsep} for embedding bipartite graphs with small separators.

\begin{theorem}\label{wellsep2}\cite{wellsep}
For every $\epsilon>0$ and positive integer $D$ there exists an $n_0=n_0(\epsilon, D)$ such that the following holds. Assume that $H$ is a bipartite graph on $n\ge n_0$ vertices which has a separator set $S$ such that $|S|=o(n),$ and every component of $H-S$ has $o(n)$ vertices. Assume further
that $\Delta(H)\le D.$ Let $G$ be an $n$-vertex graph such that $\delta(G)\ge (1/2+\epsilon)n.$ Then $H\subset G.$
\end{theorem}
 
One can observe the similarities with Theorem~\ref{fo2}. The main difference is that in Theorem~\ref{fo2} the separator set can be very large compared to the separator set in Theorem~\ref{wellsep2}. This difference requires a new embedding tool.

\subsection{Main tools for the proof}

The Regularity lemma of Szemer\'edi~\cite{SzRL} and the Blow-up lemma~\cite{KSSz1} plays a very important role in the proof. 
 While we assume familiarity with these tools, below we give a brief summary of the necessary definitions and results. 
The interested reader may consult with the survey paper by Koml\'os and Simonovits~\cite{KS} also for further details.

\begin{definition}
		Given a graph $G$ and two disjoint subsets $X, Y \subset V(G),$ the \textit{density} between $X$ and $Y$  is
	
	\begin{equation}
	d(X,Y) = \frac{e(X,Y)}{|X||Y|}.
	\end{equation}
\end{definition}

\begin{definition}
	We call a pair $(A,B)$ of disjoint vertex sets in $G$ $\eps$-regular, if for every $X\subset A$ and $Y\subset B$ we have $|d(X,Y)-d(A,B)|<\eps$, whenever $|X|>\eps|A|$ and $|Y|>\eps |B|$.
\end{definition}

We will need the so called Degree Form of the celebrated Regularity lemma of Szemer\'edi:
\begin{lemma}
	\label{reglemma_deg}
	For every $\eps>0$ there is an $M = M(\eps)$ such that if $G = (V,E)$ is any graph and $d\in[0,1]$, then there is a partition of the vertex set $V$ into $\ell+1$ clusters $W_0,W_1,\ldots,W_\ell$ and there is a subgraph $G'$ of $G$ such that
	\begin{itemize}
		\item $\ell\leq M$,
		\item $|W_0|\leq \eps|V|$,
		\item All clusters $W_i,\ i\geq 1$ are of the same size $m$
		\item $deg_{G'}(v) > deg_G(v) - (d+\eps)|V|$ for all $v\in V$,
		\item $W_i$ is an independent set in $G'$ for all $i\geq 1$,
		\item All pairs $(W_i,W_j)$ ($1\leq i<j\leq \ell$) are $\eps$-regular, each with density either 0 or at least $d$ in $G'$.
	\end{itemize}
\end{lemma}

We call $W_0$ the \textit{exceptional cluster}, $W_1, \ldots, W_{\ell}$ are the non-exceptional clusters. 

\begin{definition}
	Apply \autoref{reglemma_deg} to the graph $G=(V,E)$ with parameters $\eps$ and $d.$ We construct the {\em reduced graph} $G_r$ as follows. Its vertices are the non-exceptional clusters, and two vertices are connected if the corresponding clusters form an $\eps$-regular pair with density at least $d$.
\end{definition}

\begin{claim}\label{reduced-degree}
	Let $G = (V, E)$ be a graph of order $n$ and $\delta(G) \geq cn$ for some $c > 0.$ Let $G_r$ be
	the reduced graph of $G'$ after applying \autoref{reglemma_deg} with parameters $\eps$ and $d$.
	Then $\delta(G_r) \geq (c - \theta)\ell$, where $\theta = 2\eps + d$.
\end{claim}

\begin{definition}
	We say that the pair of disjoint vertex sets $(A,B)$ is $(\eps,\delta)$-super-regular, if it is $\eps$-regular, and $deg(a)>\delta|B|$ for all $a\in A$, and $deg(b)>\delta|A|$ for all $b\in B$.
\end{definition}

\begin{remark}\label{adas}
	\label{regular-addition}
	If $(A,B)$ is an $\eps$-regular pair, and we add $\le 2\eps|A|$ new vertices to $A$ resulting $A'$, then 
the new pair $(A',B)$ is $\eps'$-regular, where $\eps' \le 2\sqrt{\eps}.$
\end{remark}

\begin{remark}
	\label{regularity-removal}
	If $(A,B)$ is an $\eps$-regular pair with density $d$, then for any $Y\subset B,~|Y|>\eps|B|$ we have
	\begin{equation}
	|\{x\in A: |N(x)\cap Y|\leq (d-\eps)|Y|\}|\leq \eps|A|.
	\end{equation}

\end{remark}

We need the following very important lemma:

\begin{theorem}[Blow-up Lemma, \cite{KSSz1}]

Given a graph $R$ of order $r$ and positive parameters
$\delta,\Delta$, there exists a positive $\eps=\eps(\delta,\Delta,r)$ such that the following holds: Let $n_1, n_2, \ldots, n_r$ be arbitrary positive integers and replace the vertices $v_1,v_2,\ldots,v_r$ of $R$ with pairwise disjoint sets $V_1, V_2, \ldots, V_r$ of sizes $n_1, n_2, \ldots, n_r$ (blowing up). We construct two graphs on the same vertex set $V = \cup V_i$. The first graph $F$ is obtained by replacing each edge $v_iv_j$ of $R$ with the complete bipartite graph between $V_i$ and $V_j$. A sparser graph $G$ is constructed by replacing each edge $v_iv_j$ arbitrarily with an $(\eps, \delta)$-super-regular pair between $V_i$ and $V_j$. If a graph $H$ with $\Delta(H)\leq \Delta$ is embeddable into $F$, then it is also embeddable into $G$.

\end{theorem}

\begin{theorem}[Strengthening the Blow-up Lemma \cite{KSSz1}]\label{blow2}
	Assume that $n_i \leq 2n_j$ for every $1\leq i, j\leq r$. Then we can strengthen the lemma: Given $c > 0$ there are positive numbers $\eps=\eps(\delta,\Delta,r,c)$ and $\alpha = \alpha(\delta,\Delta,r, c)$ such that the Blow-up Lemma remains true if for every $i$ there are certain vertices $x$ to be embedded into $V_i$ whose images are a priori restricted to certain sets $T_x\subset V_i$	provided that each $T_x$ within a $V_i$ is of size at least $c|V_i|$, and the number of such restrictions within a $V_i$ is not more than $\alpha|V_ i|$.
	
\end{theorem}

Another important tool for the proof is the following result by Fox and Sudakov~\cite{FS}.

\begin{theorem}\label{DRCh}
Let $H$ be a bipartite graph with $n$ vertices and maximum degree $\Delta \ge 1.$  Assume that $\rho \in (0, 1)$ is a real number. If 
$G$ is a graph with $N \ge 8\Delta \rho^{-\Delta}n$ vertices and at least $\rho \binom{N}{2}$
edges, then $H \subset G.$
\end{theorem}

We are going to apply Theorem~\ref{DRCh} in the special case $\rho=1/2.$

Finally one more notation: we will write $a\ll b$ for two positive numbers $a$ and $b$ if $a$ is sufficiently smaller than $b.$

\subsection{Proof of Theorem~\ref{fo2}}

Below we present the embedding algorithm as an itemized list.
As we indicated above, certain parts of the proof of Theorem~\ref{fo2} is very similar to the proof of Theorem~\ref{wellsep2}
from~\cite{wellsep}. Hence, whenever it is possible we refer
to the claims, lemmas of~\cite{wellsep}.

\paragraph{Step 1} Assume that $H\in \mathcal{H}_{r,t}$ has $n$ vertices. Denote the separator set of $H$ by $S,$ here $|S|\le \gamma n.$ Observe that we can apply the deep result of Fox and Sudakov, Theorem~\ref{DRCh} above, for finding a copy of $H[S]$ in $G,$ since $\delta(G)>n/2.$ Let us call the uncovered part of $G$ by $\widetilde{G}$ after embedding $H[S].$
Note that $\delta(\widetilde{G})\ge (1/2+2\gamma^{1/3})n$ and $|\widetilde{G}|\ge (1-\gamma)n.$

\paragraph{Step 2} Next we apply the Degree form of the Regularity lemma for $\widetilde{G}$ with parameters $\eps$ and 
$d$ such that $\eps \ll d=\sqrt{\gamma}.$ As a result, we have $\ell  + 1$ clusters, $W_0, W_1, \ldots, W_\ell$. The exceptional cluster $W_0$ has at most $\eps n$ vertices, while all other clusters have the same size $m$. We form the reduced graph $\widetilde{G}_r$, for which we have $\delta(\widetilde{G}_r) \geq (1/2+\gamma')\ell$, where $\gamma' = 2\gamma^{1/3} - \sqrt{\gamma} - \eps > \gamma^{1/3}.$

\paragraph{Step 3} Since $\delta(\widetilde{G}_r)\ge \ell/2,$ we have a perfect (or almost perfect) matching $M$ in $\widetilde{G}_r$: 
at most one cluster will remain uncovered by $M,$ and only if $\ell$ is odd. If there is such a cluster, we add it into $W_0$.

\paragraph{Step 4} Next we transform the edges of $M$ into super-regular pairs. Given a $\delta$ with $\eps\ll \delta \ll d$ we have to remove at most $\eps m$ vertices from a cluster to make a regular pair $(\eps,\delta)$-super-regular. We discard the same number of vertices, $\eps m,$ from every non-exceptional cluster, and place the discarded vertices into $W_0$. 
Note that pairs become $\eps'$-regular with $\eps'<2\eps$, and the sizes of the clusters are $m'=m-\eps m$. For sake of simplicity we will still use $\eps$ and $m$ in the rest of the paper.

\paragraph{Step 5} The next step is to distribute the vertices of the exceptional cluster $W_0$ among the non-exceptional clusters while maintaining the super-regularity of the edges of $M.$ We also require that the resulting clusters have about the same size (i.e. we need, that $\big||W_i|-|W_j|\big|\le 3\eps m$ for all $1\leq i,j\leq \ell$). For this, we use the same procedure described in \cite{wellsep}, the reader may consult with that paper for the details, here we only sketch the argument. Let us denote the neighbor of a cluster $W_i$ in the matching $M$ by $N_M(W_i)$.


For the distribution of $W_0$ we define an auxiliary bipartite graph $J$ having vertex classes $W_0$ and $V(\widetilde{G}_r).$  Here $vW_i\in E(J)$ if $deg_G(v,N_M(W_i))\geq \delta m$.

Let $\gamma'' = 3(\gamma^{1/3} -2(\eps +d))$. We need the following lemma, which is a special case of Lemma 10 in~\cite{wellsep}. 

\begin{lemma}\label{auxiliary-lemma}
	$deg_{J}(v) \geq (\frac{1}{2} + \gamma'')\ell$ for every $v\in W_0$.
\end{lemma}

Using \autoref{auxiliary-lemma} we can distribute the vertices $W_0$ among the clusters one-by-one, essentially greedily. 
Assume that we have already found non-exceptional clusters for the first $s$ vertices of $W_0.$ Then for every $1\le i\le \ell$
we let $P_s(W_i)$ to be the number of $W_0$-vertices put into $W_i.$ 

When we look for a non-exceptional cluster for the $(s+1)$st vertex $v\in W_0,$ we take the smallest $P_s(W_i)$ value for every $W_i\in N_J(v)$ (we break ties arbitrarily). It is easy to see by the large minimum degree of $J$ that no non-exceptional
cluster will get more than $2|W_0|/\ell$ new vertices from $W_0.$ Since $|W_0|\le \eps n,$ and $n/\ell \approx m,$ we get that
$\big||W_i|-|W_j|\big|<3\eps m.$ Using Remark~\ref{adas} super-regularity of edges of $M$ is maintained, although instead of $\eps$ we have at most $2\sqrt{2\sqrt{\eps}}<3\sqrt[4]{\eps}$ here.

\paragraph{Step 6} Next we assign the components of $H-S$ to the non-exceptional clusters. We do it using a random procedure as in ~\cite{wellsep} (randomness is not necessary here, but a simple choice), i.e., the components are assigned randomly to edges of $M.$ 

The algorithm is as follows. Consider a component $A$ of $H$, it is a tree, so it is bipartite. 
Denote its vertex classes by $A_1$ and $A_2.$ Pick an edge $Q_1Q_2$ of $M$ randomly, uniformly. Let $\pi$ be a uniform random permutation on $\{1,2\}$, and assign the vertices of $A_i$ to $Q_{\pi(i)}$ for $i=1,2$. The following lemma holds (for a proof see~\cite{wellsep}).

\begin{lemma}
With positive probability the mapping algorithm assigns $n/\ell \pm \eps m/\ell$ vertices of $H$ to every edge of $\widetilde{G}_r$.
\end{lemma}

 For $x\in V(H)-S$ we let $C(x)$
denote the cluster to which $x$ is assigned to.
The algorithm immediately implies that whenever $x, y\in V(H)-S$ are adjacent, then $C(x)C(y)$ is an edge of $M.$ 

Recall that in Step 1 $H[S]$ was embedded using Theorem~\ref{DRCh}, before even applying the Regularity Lemma.
So the vertices of $S$ were not assigned to clusters but directly mapped onto vertices of $G.$ It is clear that adjacent vertices
of $S$ were mapped onto adjacent vertices of $G.$

However, for applying  Theorem~\ref{blow2}, the Blow-up lemma with restrictions, we need that if $x\in S$ is mapped onto $v\in V(G),$ $y\in V(H)-S$ is assigned to
$C(y),$ then $v$ must have many neighbors in $C(y).$ 
 
In such a case we will assign $y$ to another cluster (again, we use a procedure from~\cite{wellsep} with a few minor modifications) as follows.
Let $L$ denote those clusters in which $v$ has at least $3\gamma^{2/3} m$ neighbors. Simple counting argument shows that 
$|L|\ge \ell/2.$ Let $C(y)W_i$ denote the edge of
$M$ to which the component of $y$ was assigned. Then we locate a cluster $W_j\in L$ such that $W_j$ is adjacent to $W_i$
in $\widetilde{G}_r.$ Since $L$ has more than $\ell/2$ clusters, using the minimum degree of $\widetilde{G}_r$ 
we have at least $2\gamma^{1/3}\ell$ choices for $W_j$ in $L.$ 

Then we change the assignment of $y,$ we let $C(y)=W_j.$   
This way $v$ will have many neighbors in $C(y)$ and $C(y)$ will be adjacent to the cluster of the neighbors of $y$ in its component. 
Observe that if we locate the $W_j$ clusters as evenly as possible then we can achieve that at most  
$|S|/(2\gamma^{1/3}\ell) \le \gamma n/(2\gamma^{1/3}\ell)\le \gamma^{2/3}m$
vertices are reassigned to a particular cluster, using that $|S|\le \gamma n$ and the number of choices for a new cluster is always at least $2\gamma^{1/3}\ell.$ 
Of course, changing the assignment of vertices of $H$ does not affect regularity or super-regularity between clusters of $\widetilde{G}_r.$
  
\paragraph{Step 7} In this step we will achieve that the number of vertices assigned to a cluster is the same as the size of that cluster. As before, we will reassign some of the vertices. While in the previous step we changed the assignment of first vertices of components, this time we will work with the leaves of the components. 

Say, that $W_s$ has more vertices assigned to it than its size $|W_s|.$ Then there must be a cluster $W_i$ to which we assigned
less than $|W_i|$ vertices of $H.$ Let $W_j$ denote the neighbor of $W_s$ in the matching $M.$ If $W_jW_i$ is an edge
in $\widetilde{G}_r$ then we pick a vertex $x$ such that $C(x)=W_j$ and $deg_H(x)=D-1.$ 
This is possible since the number of last vertices (these have degree $D-1$) is at least $(1-\gamma)n/(t+D-1),$ and by Chernoff's inequality with probability $>1-1/n^2$ we assigned more than $(1-\gamma) m/(2 (t+D-1))$ to every cluster.
This implies that the number of leaves that are assigned to a particular cluster is at least
$(D-1)(1-\gamma) m/(2 (t+D-1)),$ which is much larger than $\gamma^{2/3}m.$
We reassign some of the {\it leaves} that are adjacent to $x,$ the necessary number will be assigned to $W_i.$
 
If $W_jW_i$ is not an edge, then by the minimum degree of $\widetilde{G}_r$ there must exist at least $\gamma^{1/3}\ell$ clusters $W_p$ and $W_q$ such that $W_qW_i$ and $W_jW_p$ are edges in 
$\widetilde{G}_r,$ and $W_pW_q$ is an edge in $M.$ Then the above procedure is done in two steps: first we reassign 
some vertices (always leaves) from $W_s$ to $W_p$ and then the same number of leaves from $W_p$ to $W_i.$ 
Note that the same computation works as above:
at most $\gamma^{2/3}m$ vertices are reassigned at every cluster. 

Let us remark again that reassigning the vertices of $H$ during this step does not affect regularity and super-regularity of pairs in $\widetilde{G}_r,$ unlike
when vertices of $G$ were distributed in Step 5. 

\paragraph{Step 8} Recall that the density of the regular pairs is at least $\sqrt{\gamma},$ and at most $\gamma^{2/3}m$ vertices are reassigned at every cluster. It is crucial here to look at the so called a priori restrictions for some of the vertices of $H$ before applying the Blow-up Lemma. These are the first vertices of components (since their neighbors belong to $S$ and are already mapped), and the leaves that were reassigned
during Step 7. As the number of these vertices is very small compared to $m$ in every cluster, we are able to apply the Blow-up lemma.  This finishes the proof of Theorem~\ref{fo2}.

\end{document}